\documentclass[12pt,reqno]{article}

\usepackage[usenames]{color}
\usepackage{amssymb}
\usepackage{amsmath}
\usepackage{amsthm}
\usepackage{amsfonts}
\usepackage{amscd}
\usepackage{graphicx}

\usepackage[colorlinks=true,
linkcolor=webgreen,
filecolor=webbrown,
citecolor=webgreen]{hyperref}

\definecolor{webgreen}{rgb}{0,.5,0}
\definecolor{webbrown}{rgb}{.6,0,0}

\usepackage{color}
\usepackage{fullpage}
\usepackage{float}

\usepackage{graphics}
\usepackage{latexsym}
\usepackage{epsf}
\usepackage{breakurl}

\newcommand{\seqnum}[1]{\href{https://oeis.org/#1}{\rm \underline{#1}}}

\begin{document}

\theoremstyle{plain}
\newtheorem{theorem}{Theorem}
\newtheorem{corollary}[theorem]{Corollary}
\newtheorem{lemma}[theorem]{Lemma}
\newtheorem{proposition}[theorem]{Proposition}

\theoremstyle{definition}
\newtheorem{definition}[theorem]{Definition}
\newtheorem{example}[theorem]{Example}
\newtheorem{conjecture}[theorem]{Conjecture}

\theoremstyle{remark}
\newtheorem{remark}[theorem]{Remark}

\title{Runs in Paperfolding Sequences}

\author{Jeffrey Shallit\footnote{Research supported by a grant from NSERC,
2024-03725.}\\
School of Computer Science\\
University of Waterloo\\
Waterloo, ON N2L 3G1 \\
Canada\\
\href{mailto:shallit@uwaterloo.ca}{\tt shallit@uwaterloo.ca}}

\maketitle

\begin{abstract}
The paperfolding sequences form an uncountable class of infinite
sequences over the alphabet $\{ -1, 1 \}$ that describe the sequence
of folds arising from iterated folding of a piece of paper,
followed by unfolding. In this note we observe that the sequence
of run lengths in such a sequence, as well as the starting and 
ending positions of the $n$'th run,
is $2$-synchronized and hence computable by a finite automaton.
As a specific consequence, we obtain the recent results of Bunder, Bates, and
Arnold, in much more generality, via a different approach.  We also prove results about the
critical exponent and subword complexity of these run-length sequences.
\end{abstract}

\section{Introduction}

Paperfolding sequences are sequences over the alphabet
$\{ -1, 1\}$ that arise from the iterated folding of a piece of paper,
introducing a hill ($+1$) or valley ($-1$) at each fold.  They are
admirably discussed, for example, in \cite{Davis&Knuth:1970,Dekking&MendesFrance&vanderPoorten:1982}.

The formal definition of a paperfolding sequence is based on a
(finite or infinite) sequence of {\it unfolding instructions} $\bf f$.
For finite sequences $\bf f$ 
we define
\begin{align}
P_\epsilon &= \epsilon \nonumber\\
P_{{\bf f} a} &= (P_{\bf f}) \  a \ ({-P_{{\bf f}}^R}) \label{fund}
\end{align}
for $a \in \{ -1, 1\}$ and ${\bf f} \in \{-1, 1\}^*$.
Here $\epsilon$ denotes the empty sequence of length $0$,
$-x$ changes the sign of each element of a sequence $x$, and
$x^R$ reverses the order of symbols in a sequence $x$.
An easy induction now shows that $|P_{\bf f}| = 2^{|{\bf f}|} - 1$, where
$|x|$ means the length, or number of symbols, of a sequence $x$.

Now let ${\bf f} = f_0 f_1 f_2 \cdots$ be an infinite sequence in
$\{-1, 1\}^\omega$.  It is easy to see that 
$P_{f_0 f_1 \cdots f_n}$ is a prefix of
$P_{f_0 f_1 \cdots f_{n+1}}$ for all $n \geq 0$, so there is a unique
infinite sequence of which all the $P_{f_0 f_1 \cdots f_n}$
are prefixes; we call this infinite sequence $P_{\bf f}$.

As in the previous paragraph, 
we always index the unfolding instructions starting
at $0$:  ${\bf f} = f_0 f_1 f_2 \cdots$.  Also by convention the
paperfolding sequence itself is indexed starting at $1$:
$P_{\bf f} = p_1 p_2 p_3 \cdots$.
With these conventions we immediately see that $P_{\bf f} [2^n] = p_{2^n} = f_n$
for $n \geq 0$.  Since there are a countable infinity of choices 
between $-1$ and $1$ for each unfolding instructions, there are
uncountably many infinite paperfolding sequences.

As an example let us consider the most famous such sequence, the
{\it regular paperfolding sequence}, where the
sequence of unfolding instructions is $1^\omega = 111\cdots$.  
Here we have, for example,
\begin{align*}
P_1 &= 1 \\
P_{11} &= 1 \, 1 \, (-1) \\
P_{111} &= 1 \, 1 \, (-1) \, 1 \, 1 \, (-1) \, (-1) .
\end{align*}
The first few values of the limiting infinite paperfolding
sequence $P_{1^\omega} [n]$ are given in Table~\ref{tab1}.
\begin{table}[htb]
\begin{center}
\begin{tabular}{c|ccccccccccccccccc} 
$n$ & 1 & 2 &  3 & 4 & 5 & 6 & 7 & 8 & 9 & 10 & 11 & 12 & 13 & 14 & 15 & 16 & $\cdots$\\
\hline
$P_{1^\omega} [n]$ & 1& 1&$-1$& 1& 1&$-1$&$-1$& 1& 1& 1&$-1$&$-1$& 1&$-1$&$-1$ & 1& $\cdots$
\end{tabular}
\end{center}
\caption{The regular paperfolding sequence.}
\label{tab1}
\end{table}

The paperfolding sequences have a number of interesting properties
that have been explored in a number of papers.  In addition to the
papers 
\cite{Davis&Knuth:1970,Dekking&MendesFrance&vanderPoorten:1982} already
cited, the reader can also see
Allouche \cite{Allouche:1992},
Allouche and Bousquet-M\'elou 
\cite{Allouche&Bousquet-Melou:1994a,Allouche&Bousquet-Melou:1994b},
and Go\v{c} et al.~\cite{Goc&Mousavi&Schaeffer&Shallit:2015}, to name 
just a few.

Recently Bunder et al.~\cite{Bunder&Bates&Arnold:2024} explored
the sequence of lengths of runs of the regular paperfolding sequence,
and proved some theorems about them.  Here by a ``run'' we mean a maximal
block of consecutive identical values.  Runs and run-length
encodings are a long-studied feature of sequences; see,
for example, \cite{Golomb:1966}.  
The run lengths $R_{1111}$ for the finite paperfolding sequence $P_{1111}$,
as well as the starting positions $S_{1111}$ and ending
positions $E_{1111}$ of the $n$'th run, are
given in Table~\ref{tab2}.  

\begin{table}[htb]
\begin{center}
\begin{tabular}{c|ccccccccccccccc} 
$n$ & 1 & 2 &  3 & 4 & 5 & 6 & 7 & 8 & 9 & 10 & 11 & 12 & 13 & 14 & 15 \\
\hline
$P_{1111} [n] $ & 1& 1&$-1$& 1& 1&$-1$&$-1$& 1& 1& 1&$-1$&$-1$& 1&$-1$&$-1$  \\
$R_{1111} [n] $ & 2&1&2&2&3&2&1&2& & & & & & &   \\
$S_{1111} [n] $ & 1& 3& 4& 6& 8&11&13&14&  &  &  &  &  &  &   \\
$E_{1111} [n] $ & 2& 3& 5& 7&10&12&13&15&  &  &  &  &  &  &   \\
\end{tabular}
\end{center}
\caption{Run lengths of the regular paperfolding sequence.}
\label{tab2}
\end{table}

As it turns out, however, {\it much\/} more general results,
applicable to {\it all\/} paperfolding sequences,
can be proven rather simply, in some cases making use of the
{\tt Walnut} theorem-prover \cite{Mousavi:2016}.  As shown in
\cite{Shallit:2023}, to use {\tt Walnut} it suffices to state a claim in
first-order logic, and then the prover can rigorously determine its truth or
falsity.

In order to use {\tt Walnut} to study the run-length sequences, these
sequences must be computable by a finite automaton (``automatic'').
Although the
paperfolding sequences themselves have this property (as shown,
for example, in \cite{Goc&Mousavi&Schaeffer&Shallit:2015}), 
there is no reason, a priori, to expect that the sequence of run lengths
will also have the property.
For example, the sequence of runs of the Thue-Morse sequence
${\bf t} = 0110100110010110\cdots$ is
$12112221121\cdots$, fixed point of the morphism $1 \rightarrow 121$,
$2 \rightarrow 12221$ \cite{Allouche&Arnold&Berstel&Brlek&Jockusch&Plouffe&Sagan:1995}, and is known to {\it not\/} be automatic \cite{Allouche&Allouche&Shallit:2006}.

The starting and ending positions 
of the $n$'th run are integer sequences.
In order to use {\tt Walnut} to study these, we would need these sequences
to be {\it synchronized\/} (see \cite{Shallit:2021}); that is, there would
need to be an automaton that reads the integers $n$ and $x$ in parallel and accepts
if $x$ is the starting (resp., ending) position of the $n$'th run.
But there is no reason, a priori, that the starting and ending 
positions of the $n$'th run of an arbitrary automatic sequence should be
synchronized. Indeed, if this were the case, and the length of runs were
bounded, then the length of these runs would
always be automatic, which as we have just seen is not the case for the
Thue-Morse sequence.

However, as we will see,
there is a single finite automaton that can compute
the run sequence $R_{\bf f}$ for
{\it all\/} paperfolding sequences simultaneously,
and the same thing applies to the sequences $S_{\bf f}$ and $E_{\bf f}$
of starting
and ending positions respectively.

In this paper we use these ideas to study the run-length sequences of
paperfolding sequences, explore their critical exponent and subword
complexity, and generalize the results of Bunder et al.~\cite{Bunder&Bates&Arnold:2024} on the continued fraction of a specific real number to uncountably
many real numbers.

\section{Automata for the starting and ending positions of runs}

We start with a basic result with a simple induction proof.
\begin{proposition}
Let $\bf f$ be a finite sequence of unfolding instructions of
length $n$.  Then the corresponding run-length sequence 
$R_{\bf f}$, as well as $S_{\bf f}$ and $E_{\bf f}$,
has length $2^{n-1}$.
\end{proposition}

\begin{proof}
The result is clearly true for $n=1$.  Now suppose
${\bf f}$ has length $n+1$ and write 
${\bf f} = {\bf g} a$ for $a \in \{ -1,1 \}$.
For the induction step,
we use Eq.~\eqref{fund}.  From it, we see that there are
$2^{n-1}$ runs in $P_{\bf g}$ and in $-P_{\bf g}^R$.
Since the last symbol of $P_{\bf g}$ is the negative of
the first symbol of $-P_{\bf g}^R$,
introducing $a$ between them extends the length of
one run, and doesn't affect the other.  Thus we do not introduce
a new run, nor combine two existing runs into one.
Hence the number of runs in $P_{\bf f} $ is $2^n$, as desired.
\end{proof}

\begin{remark}
Bunder et al.~\cite{Bunder&Bates&Arnold:2024} proved the same result
for the specific case of the regular paperfolding sequence.
\end{remark}

Next, we find automata for the starting and ending positions of
the runs.  Let us start with the starting positions.

The desired automaton $\tt sp$ takes three inputs in parallel.
The first input is a finite
sequence $\bf f$ of unfolding instructions, the second is the number $n$ written
in base $2$, and the third is some number $x$, also expressed in base $2$.
The automaton accepts if and only if $x = S_{\bf f} [n]$.

Normally we think of the unfolding
instructions as over the alphabet $\{ -1, 1 \}$, but it is useful to
be more flexible and also allow $0$'s, but only at the end;
these $0$'s are essentially disregarded.  
We need this because the parallel reading of inputs requires
that all three inputs be of the same length.  
Thus, for example, the sequences $-1, 1, 1, 0$ and $-1, 1, 1$ are
considered to specify the same paperfolding sequence, while
$-1, 0, 1, 1$ is not considered a valid specification.

Because we choose to let $f_0$ be the first symbol of the unfolding instructions, it is also useful to require that the inputs $n$ and $x$ mentioned above
be represented with the {\it least-significant digit first}.
In this representation, we allow an unlimited number of trailing zeros.

Finally, although we assume that $S_{\bf f}$ is indexed starting at position $1$,
it is useful to define $S_{\bf f}[0] = 0$ for all finite unfolding
instruction sequences $\bf f$.

To find the automaton computing the starting positions of runs, we use
a guessing procedure described in \cite{Shallit:2023}, based on a variant
of the Myhill-Nerode theorem.  Once a candidate automaton is guessed,
we can rigorously verify its correctness with {\tt Walnut}.

We will need one {\tt Walnut} automaton already introduced
in \cite{Shallit:2023}: {\tt FOLD}, and another one that we can define
via a regular expression.
\begin{itemize}
\item {\tt FOLD} takes two inputs, $\bf f$ and $n$.  If $n$ is
in the range $1 \leq n < 2^{|{\bf f}|}$, then it returns the 
$n$'th term of the paperfolding sequence specified by $f$.
\item {\tt lnk} takes two inputs, $f$ and $x$. It accepts if
$f$ is the valid code of a paperfolding sequence (that is, no
$0$'s except at the end) and $x$ is $2^t-1$, where $t$ is the
length of $f$ (not counting $0$'s at the end).
It can be created using the {\tt Walnut} command
\begin{verbatim}
reg lnk {-1,0,1} {0,1} "([-1,1]|[1,1])*[0,0]*":
\end{verbatim}

\end{itemize}

Our guessed automaton {\tt sp} has $17$ states.  We must
now verify that it is correct.  To do so we need to verify the following
things:
\begin{enumerate}
\item The candidate automaton {\tt sp} computes
 a partial function.  More precisely, for a given $\bf f$ and $n$, at most one
 input of the form $({\bf f},n,x)$ is accepted.
\item {\tt sp} accepts $({\bf f},0,0)$.
\item {\tt sp} accepts $({\bf f},1,1)$ provided $|{\bf f}| \geq 1$.
\item There is an $x$ such that {\tt sp} accepts $({\bf f},2^{|{\bf f}|-1},x)$.
\item {\tt sp} accepts no input of the form
$({\bf f},n,x)$ if $n > 2^{|{\bf f}|-1}$.
\item If {\tt sp} accepts $({\bf f},2^{|{\bf f}|-1},x)$ then
the symbols $P_{\bf f}[t]$ for $x \leq t < 2^{|{\bf f}|}$ are all the same.
\item Runs are nonempty: 
if {\tt sp} accepts $({\bf f},n-1,y)$ and $({\bf f},n,z)$ then $y<z$.
\item And finally, we check that if ${\tt sp}$ accepts
$({\bf f},n,x)$, then $x$ is truly the starting position of
the $n$'th run.  This means that all the symbols from the
starting position of the $(n-1)$'th run to $x-1$ are the
same, and different from $P_{\bf f}[x]$.
\end{enumerate}

We use the following {\tt Walnut} code to check each of these.  A brief
review of {\tt Walnut} syntax may be useful:
\begin{itemize}
\item {\tt ?lsd\_2} specifies that all numbers are represented with the
least-significant digit first, and in base $2$;
\item {\tt A} is the universal quantifier $\forall$ and {\tt E} is the existential quantifier $\exists$;
\item {\tt \&} is logical {\tt AND}, {\tt |} is logical {\tt OR}, {\tt \char'127} is logical {\tt NOT},
{\tt =>} is logical implication, {\tt <=>} is logical IFF, and
{\tt !=} is inequality;
\item {\tt eval} expects a quoted string representing a first-order assertion
with no free (unbound) variables, and returns {\tt TRUE} or {\tt FALSE};
\item {\tt def} expects a quoted string representing a first-order assertion
$\varphi$
that may have free (unbound) variables, and computes an automaton accepting
the representations of those tuples of variables that make $\varphi$ true,
which can be used later.
\end{itemize}
\begin{verbatim}
eval tmp1 "?lsd_2 Af,n ~Ex,y x!=y & $sp(f,n,x) & $sp(f,n,y)":
# check that it is a partial function
eval tmp2 "?lsd_2 Af,x $lnk(f,x) => $sp(f,0,0)":
# check that 0th run is at position 0; the lnk makes sure that
# the format of f is correct (doesn't have 0's in the middle of it.)
eval tmp3 "?lsd_2 Af,x ($lnk(f,x) & x>=1) => $sp(f,1,1)":
# check if code specifies nonempty string then first run is at position 1
eval tmp4 "?lsd_2 Af,n,z ($lnk(f,z) & z+1=2*n) => Ex $sp(f,n,x)":
# check it accepts n = 2^{|f|-1}
eval tmp5 "?lsd_2 Af,n,z ($lnk(f,z) & z+1<2*n) => ~Ex $sp(f,n,x)":
# check that it accepts no n past 2^{|f|-1}
eval tmp6 "?lsd_2 Af,n,z,x ($lnk(f,z) & 2*n=z+1 & $sp(f,n,x)) 
   => At (t>=x & t<z) => FOLD[f][x]=FOLD[f][t]":
# check last run is right and goes to the end of the finite
# paperfolding sequence specified by f
eval tmp7 "?lsd_2 Af,n,x,y,z ($lnk(f,z) & $sp(f,n-1,x) & 
   $sp(f,n,y) & 1<=n & 2*n<=z+1) => x<y":
# check that starting positions form an increasing sequence
eval tmp8 "?lsd_2 Af,n,x,y,z,t ($lnk(f,z) & n>=2 & $sp(f,n-1,y) &
   $sp(f,n,x) & x<=z & y<=t & t<x) => FOLD[f][x]!=FOLD[f][t]":
# check that starting position code is actually right
\end{verbatim}

{\tt Walnut} returns {\tt TRUE} for all of these, which gives us a proof by
induction on $n$ that indeed $x_n = S_{\bf f}[n]$.  

From the automaton for starting positions of runs, we can obtain the automaton
for ending positions of runs, {\tt ep}, using the following {\tt Walnut}
code:
\begin{verbatim}
def ep "?lsd_2 Ex $lnk(f,x) & ((2*n<=x-1 & $sp(f,n+1,z+1)) | 
   (2*n-1=x & z=x))":
\end{verbatim}

Thus we have proved the following result.
\begin{theorem}
There is a synchronized automaton of $17$ states {\tt sp} computing
$S_{\bf f}[n]$
and one of $13$ states {\tt ep} computing
$E_{\bf f}[n]$, for all paperfolding
sequences simultaneously.
\end{theorem}

Using the automaton {\tt ep}, we are now able to prove the following new theorem. 
Roughly speaking, it says
that the ending position of the $n$'th run for the unfolding instructions
$\bf f$ is $2n - \epsilon_n$, where
$\epsilon_n \in \{0, 1 \}$, and we can compute $\epsilon_n$ by
looking at a sequence of unfolding instructions closely related to
$\bf f$.

\begin{theorem}
Let $\bf f$ be a finite sequence of unfolding instructions, of
length at least $2$.
Define a new sequence $\bf g$ of unfolding instructions as follows:
\begin{equation}
{\bf g} := \begin{cases}
	1 \ (-x), & \text{if ${\bf f} = 11x$;} \\
	(-1) \ (-x), & \text{if ${\bf f} = 1 (-1) x$;} \\
	(-1) \ x, & \text{if ${\bf f} = (-1) 1 x $; } \\
	1 \ x, & \text{if ${\bf f} = (-1) (-1) x$}.
	\end{cases}
\label{eq1}
\end{equation}
Then 
\begin{equation}
 E_{\bf f}[n] + \epsilon_n = 2n
\label{2n}
\end{equation}
for $1 \leq n < 2^{n-1}$,
where 
$$\epsilon_n = \begin{cases}
	0, & \text{if $P_{\bf g}[n] = 1$;} \\
	1, & \text{if $P_{\bf g}[n]=-1$.}
	\end{cases}
$$
Furthermore, if $\bf f$ is an infinite set of unfolding instructions,
then Eq.~\eqref{2n} holds for all $n \geq 1$.

\end{theorem}

\begin{proof}
We prove this using {\tt Walnut}. First, we need an automaton
{\tt assoc} that takes two inputs $\bf f$ and $\bf g$ in parallel,
and accepts if $\bf g$ is defined as in Eq.~\eqref{eq1}.
This automaton is depicted in Figure~\ref{fig3},
and correctness is left to the reader.
Now we use the following {\tt Walnut} code.
\begin{verbatim}
eval thm3 "?lsd_2 Af,g,y,n,t ($lnk(g,y) & $assoc(f,g) & y>=1 &
   n<=y & n>=1 & $ep(f,n,t)) =>
   ((FOLD[g][n]=@-1 & t+1=2*n)|(FOLD[g][n]=@1 & t=2*n))":
\end{verbatim}
And {\tt Walnut} returns {\tt TRUE}.
\begin{figure}[htb]
\begin{center}
\includegraphics[width=5.5in]{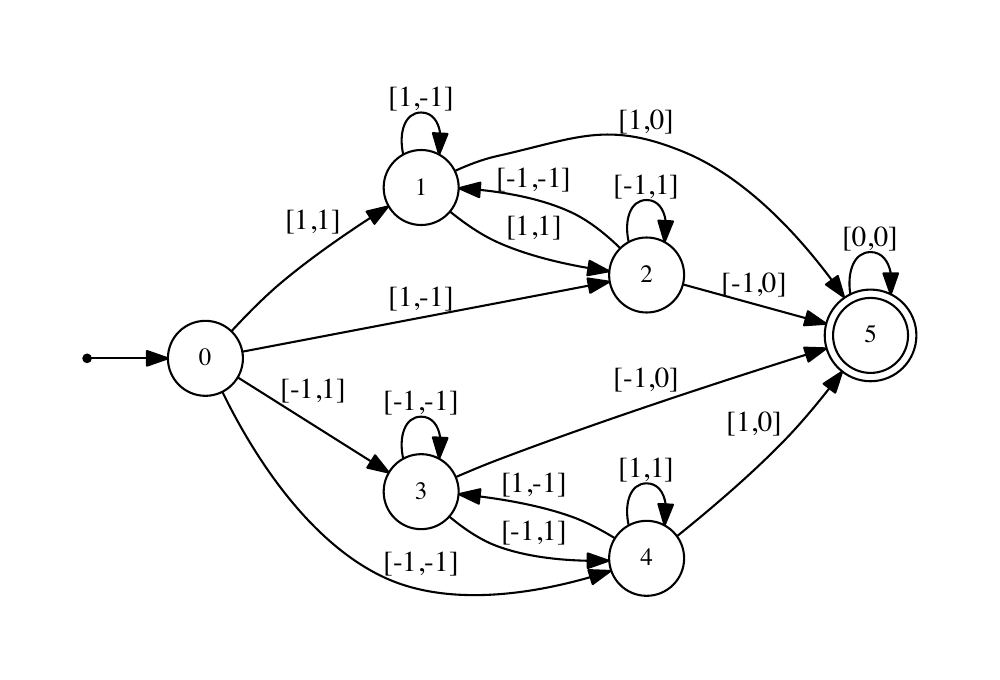}
\end{center}
\caption{The automaton {\tt assoc}.}
\label{fig3}
\end{figure}
\end{proof}

\section{Automaton for the sequence of run lengths}

Next we turn to the sequence of run lengths itself.
We can compute these from the automata for {\tt ep} and {\tt sp}.
\begin{verbatim}
def rl "?lsd_2 Ex,y $sp(f,n,x) & $ep(f,n,y) & z=1+(y-x)":
\end{verbatim}

\begin{proposition}
For all finite and infinite sequences of paperfolding instructions,
the only run lengths are $1,2,$ or $3$.
\label{prop4}
\end{proposition}

\begin{proof}
It suffices to prove this for the finite paperfolding sequences.
\begin{verbatim}
def prop4 "?lsd_2 Af,n,x,z ($lnk(f,x) & 1<=n & 2*n<=x+1
   & $rl(f,n,z)) => (z=1|z=2|z=3)":	
\end{verbatim}
And {\tt Walnut} returns {\tt TRUE}.
\end{proof}

\begin{remark}
Proposition~\ref{prop4} was proved by Bunder et al.~\cite{Bunder&Bates&Arnold:2024} for the specific case of the regular paperfolding sequence.
\end{remark}

We now use another feature of {\tt Walnut}, which is that we can turn
a synchronized automaton 
computing a function of finite range into an automaton 
returning the value of the function. 
The following code
\begin{verbatim}
def rl1 "?lsd_2 $rl(f,n,1)":
def rl2 "?lsd_2 $rl(f,n,2)":
def rl3 "?lsd_2 $rl(f,n,3)":
combine RL rl1=1 rl2=2 rl3=3:
\end{verbatim}
computes an automaton {\tt RL} of two inputs $\bf f$ and $n$, and
returns the value of the run-length sequence at index $n$
(either $1$, $2$, or $3$) for the unfolding instructions
$\bf f$.  This automaton has $31$ states.

We now turn to examining the factors of the run-length sequences of
paperfolding sequence.  Recall that a factor is a contiguous block
sitting inside a large sequence.  We start with overlaps.

Recall that an {\it overlap} is a string of the form $axaxa$, where $a$ is
a single letter, and $x$ is a possibly empty string.  For example,
the word {\tt entente} is an overlap from French. We now prove that the
sequence of run lengths in a paperfolding sequence contains no
overlaps.  

\begin{theorem}
The sequence of run lengths corresponding to every
finite or infinite paperfolding sequence is overlap-free.
\end{theorem}

\begin{proof}
It suffices to prove the result for every finite paperfolding sequence.
We can do this is as follows:
\begin{verbatim}
def chk_over "?lsd_2 ~Ef,i,n,x $lnk(f,x) & x>=1 & i>=1 & n>=1
   & i+2*n<=(x+1)/2 & At (t<=n) => RL[f][i+t]=RL[f][i+n+t]":
# asserts no overlaps
\end{verbatim}
And {\tt Walnut} returns {\tt TRUE}.
\end{proof}

We now consider {\tt squares}, that is, blocks of the form
$zz$, where $z$ is a nonempty sequence.
\begin{theorem}
The only possible squares occurring in the run lengths of a
paperfolding sequence
are $22$, $123123$, and $321321$.
\end{theorem}

\begin{proof}
We start by showing that the only squares are of order $1$ or $3$.
\begin{verbatim}
def chk_sq1 "?lsd_2 Af,i,n,x ($lnk(f,x) & x>=1 & i>=1 & n>=1
   & i+2*n-1<=(x+1)/2 & At (t<n) => RL[f][i+t]=RL[f][i+n+t]) => (n=1|n=3)":
\end{verbatim}
Next we check that the only square of order $1$ is $22$.
\begin{verbatim}
def chk_sq2 "?lsd_2 Af,x,i ($lnk(f,x) & x>=1 & i>=1 &
   i+1<=(x+1)/2 & RL[f][i]=RL[f][i+1]) => RL[f][i]=@2":
\end{verbatim}
Finally, we check that the only squares of order $3$ are $123123$ and
$321321$.
\begin{verbatim}
def chk_sq3 "?lsd_2 Af,x,i ($lnk(f,x) & x>=1 & i>=1 & 
   i+5<=(x+1)/2 & RL[f][i]=RL[f][i+3] & RL[f][i+1]=RL[f][i+4]
   & RL[f][i+2]=RL[f][i=5]) => ((RL[f][i]=@1 & RL[f][i+1]=@2
   & RL[f][i+2]=@3)|(RL[f][i]=@3 & RL[f][i+1]=@2 & RL[f][i+2]=@1))":
\end{verbatim}
\end{proof}

\begin{proposition}
In every finite paperfolding sequence formed by $7$ or more unfolding
instructions, the squares $22$, $123123$, and $321321$ are all present in
the run-length sequence.
\end{proposition}

We now turn to palindromes.
\begin{theorem}
The only palindromes that can occur
in the run-length sequence of a paperfolding
sequence are $1,2,3, 22, 212, 232, 12321, $ and $32123$.
\end{theorem}

\begin{proof}
It suffices to check the factors of the run-length sequences of length
at most $7$.  These correspond to factors of length at most $2+3\cdot 7 = 23$,
and by the bounds on the ``appearance'' function given
in Theorem~\cite[Thm 12.2.2]{Shallit:2023}, to guarantee we have seen
all of these factors, it suffices to look at prefixes of paperfolding
sequences of length at most $13 \cdot 23 = 299$.
(Also see \cite{Burns:2022}.)
Hence it suffices
to look at all $2^9$ finite paperfolding sequences of length $2^9 - 1 = 511$
specified by instructions of length $9$.  When we do this, the only
palindromes we find are those in the statement of the theorem.
\end{proof}

Recall that the {\it subword complexity} of an infinite sequence is the
function that counts, for each $n \geq 0$, the number of distinct factors
of length $n$ appearing in it.  The subword complexity of the paperfolding
sequences was determined by Allouche 
\cite{Allouche:1992}.

\begin{theorem}
The subword complexity of the run-length sequence of an infinite paperfolding
sequence is $4n+4$ for $n \geq 6$.
\end{theorem}

\begin{proof}
First we prove that if $x$ is a factor of a run-length sequence, and
$|x| \geq 2$, then $xa$ is a factor of the same sequence for at most
two different $a$.
\begin{verbatim}
def faceq "?lsd_2 At (t<n) => RL[f][i+t]=RL[f][j+t]":
eval three "?lsd_2 Ef,i,j,k,n n>=2 & i>=1 & RL[f][i+n]=@1 &
   RL[f][j+n]=@2 & RL[f][k+n]=@3 & $faceq(f,i,j,n) & $faceq(f,j,k,n)":
\end{verbatim}

Next we prove that if $|x| \geq 5$, then exactly four factors of a run-length
sequence are right-special (have a right extension by two different letters). 
\begin{verbatim}
def rtspec "?lsd_2 Ej,x $lnk(f,x) & i+n<=x & i>=1 &
   $faceq(f,i,j,n) & RL[f][i+n]!=RL[f][j+n]":
eval nofive "?lsd_2 ~Ef,i,j,k,l,m,n n>=5 & i<j & j<k & k<l
   & l<m & $rtspec(f,i,n) & $rtspec(f,j,n) & $rtspec(f,k,n) &
   $rtspec(f,l,n) & $rtspec(f,m,n)":
eval four "?lsd_2 Af,n,x ($lnk(f,x) & x>=127 & n>=6 &
   13*n<=x) => Ei,j,k,l i>=1 & i<j & j<k & k<l &
   $rtspec(f,i,n) & $rtspec(f,j,n) & $rtspec(f,k,n) & $rtspec(f,l,n)":
\end{verbatim}
Here {\tt nofive} shows that no length 5 or larger has five
or more right-special factors of that length, and every length $6$ or larger
has exactly four such right-special factors.  Here we have used
\cite[Thm.~12.2.2]{Shallit:2023}, which guarantees that every factor
of length $n$ of a paperfolding sequence can be found in a prefix
of length $13n$.   Thus we see if there are $t$ factors of length $n \geq 6$
then there are $t+4$ factors of length $n+1$:  the $t$ arising from those
that can be extended in exactly one way to the right, and the $4$ additional
from those that have two extensions.

Since there are $28$ factors of every run-length sequence of length $6$ 
(which we can check just by enumerating them, again
using \cite[Thm.~12.2.2]{Shallit:2023}), the result now
follows by a trivial induction.
\end{proof}

\section{The regular paperfolding sequence}

In this section we specialize everything we have done so far to the
case of a single infinite
paperfolding sequence, the so-called regular paperfolding
sequence, where the folding instructions are $1^\omega = 111\cdots$.
In \cite{Bunder&Bates&Arnold:2024},
the sequence $2122321231232212\cdots$
of run lengths 
for the regular paperfolding
sequence was called $g(n)$, and
the sequence 
$2, 3, 5, 7, 10, 12, 13, 15, 18, 19, 21,\ldots$
of ending positions of runs was called $h(n)$.
We adopt their notation.
Note that $g(n)$ forms sequence \seqnum{A088431} in the
On-Line Encyclopedia of Integer Sequences (OEIS) \cite{oeis}, while
the sequence of starting positions of runs
$1, 3, 4, 6, 8, 11, 13, 14, 16,\ldots$ is \seqnum{A371594}.

In this case we can compute an automaton computing
the $n$'th term of the run length sequence $g(n)$ as follows:
\begin{verbatim}
reg rps {-1,0,1} {0,1} "[1,1]*[0,0]*":
def runlr1 "?lsd_2 Ef,x $rps(f,x) & n>=1 & n<=x/2 & RL[f][n]=@1":
def runlr2 "?lsd_2 Ef,x $rps(f,x) & n>=1 & n<=x/2 & RL[f][n]=@2":
def runlr3 "?lsd_2 Ef,x $rps(f,x) & n>=1 & n<=x/2 & RL[f][n]=@3":
combine RLR runlr1=1 runlr2=2 runlr3=3:
\end{verbatim}

The resulting automaton is depicted in Figure~\ref{fig4}.
\begin{figure}[htb]
\begin{center}
\includegraphics[width=5.5in]{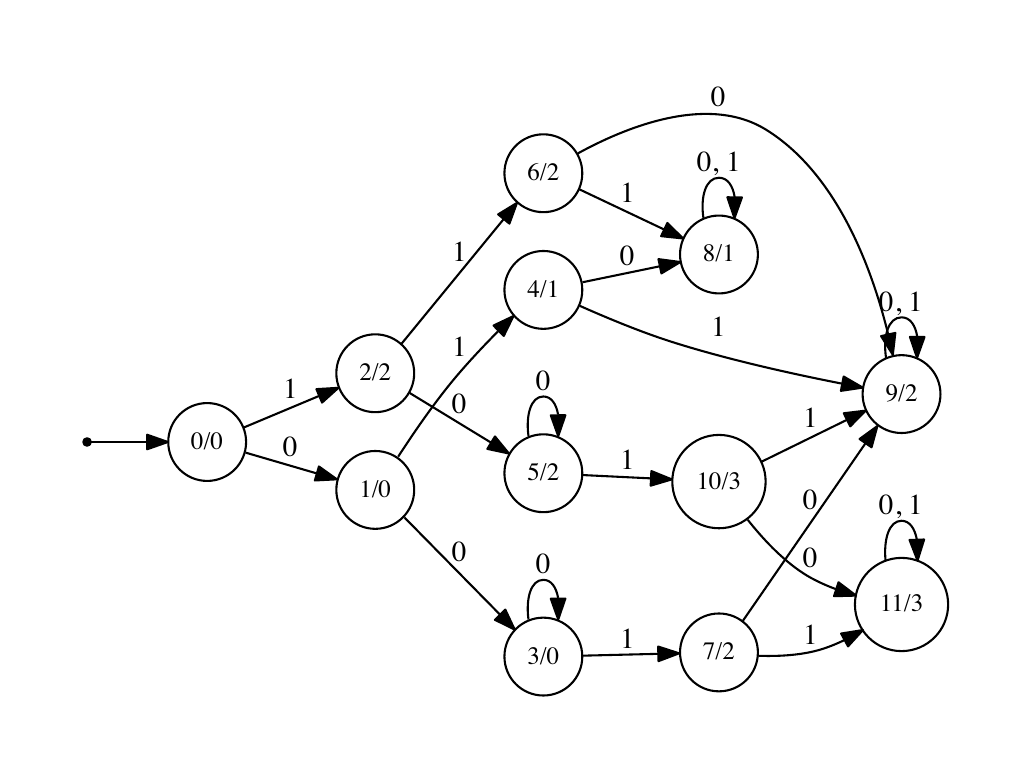}
\end{center}
\caption{The lsd-first automaton {\tt RLR}.}
\label{fig4}
\end{figure}

Casual inspection of this automaton immediately proves many
of the results of
\cite{Bunder&Bates&Arnold:2024}, such as 
their multi-part Theorems 2.1 and 2.2. To name just one example,
the sequence $g(n)$ takes the value $1$ iff
$n \equiv 2, 7$ (mod $8$).  For their other results, we can use
{\tt Walnut} to prove them.

We can also specialize {\tt sp} and {\tt ep} to the case of the regular
paperfolding sequence, as follows:
\begin{verbatim}
reg rps {-1,0,1} {0,1} "[1,1]*[0,0]*":
def sp_reg "?lsd_2 (n=0&z=0) | Ef,x $rps(f,x) & n>=1 & n<=x/2 & $sp(f,n,z)":
def ep_reg "?lsd_2 (n=0&z=0) | Ef,x $rps(f,x) & n>=1 & n<=x/2 & $ep(f,n,z)":
\end{verbatim}
These automata are depicted in Figures~\ref{fig7} and \ref{fig8}.
\begin{figure}[htb]
    \centering
    \begin{minipage}{0.47\textwidth}
        \centering
        \includegraphics[height=1.7in]{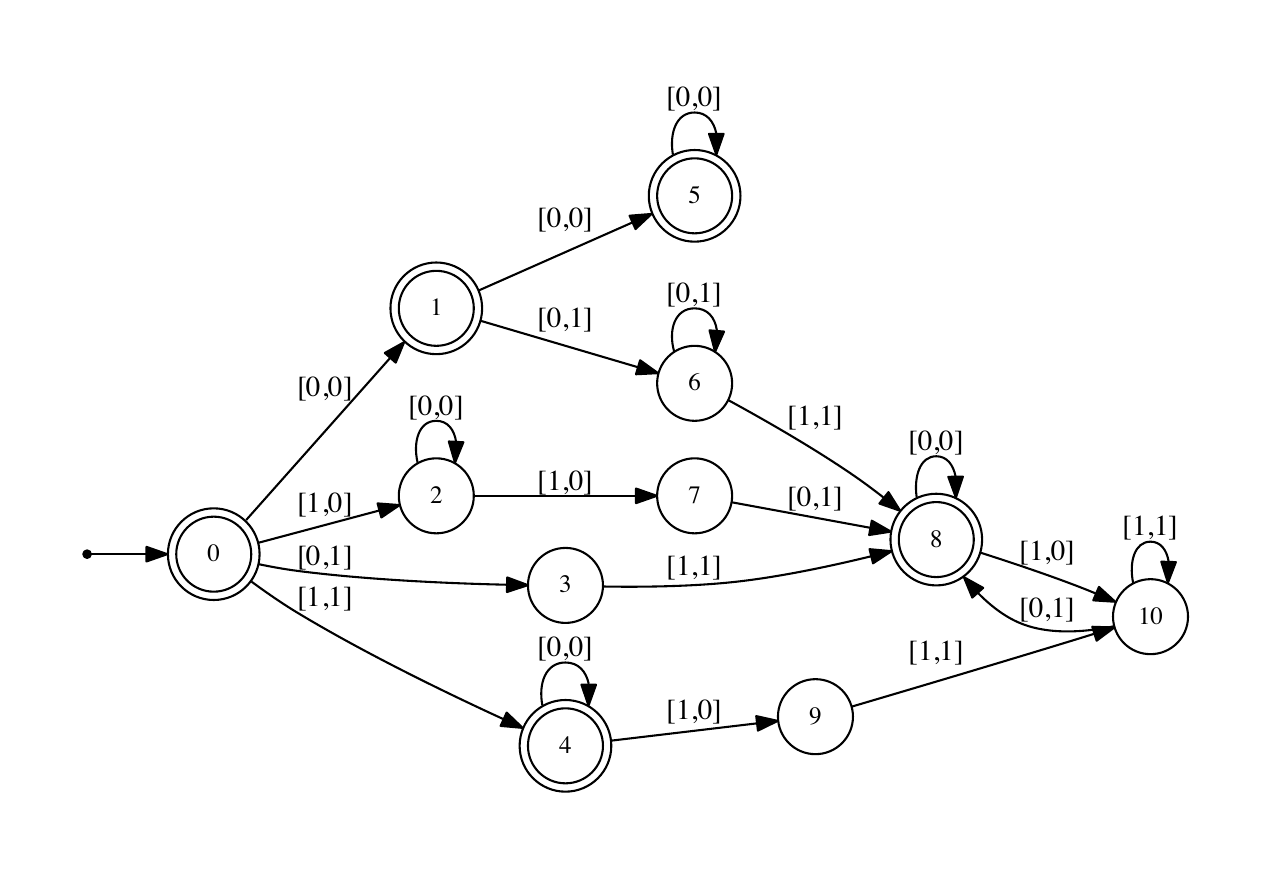}
  \caption{Synchronized automaton {\tt sp\_reg} for starting positions
of runs of the regular paperfolding sequence.}
        \label{fig7}
    \end{minipage} 
    \quad
    \begin{minipage}{0.47\textwidth}
        \centering
        \includegraphics[height=1.5in]{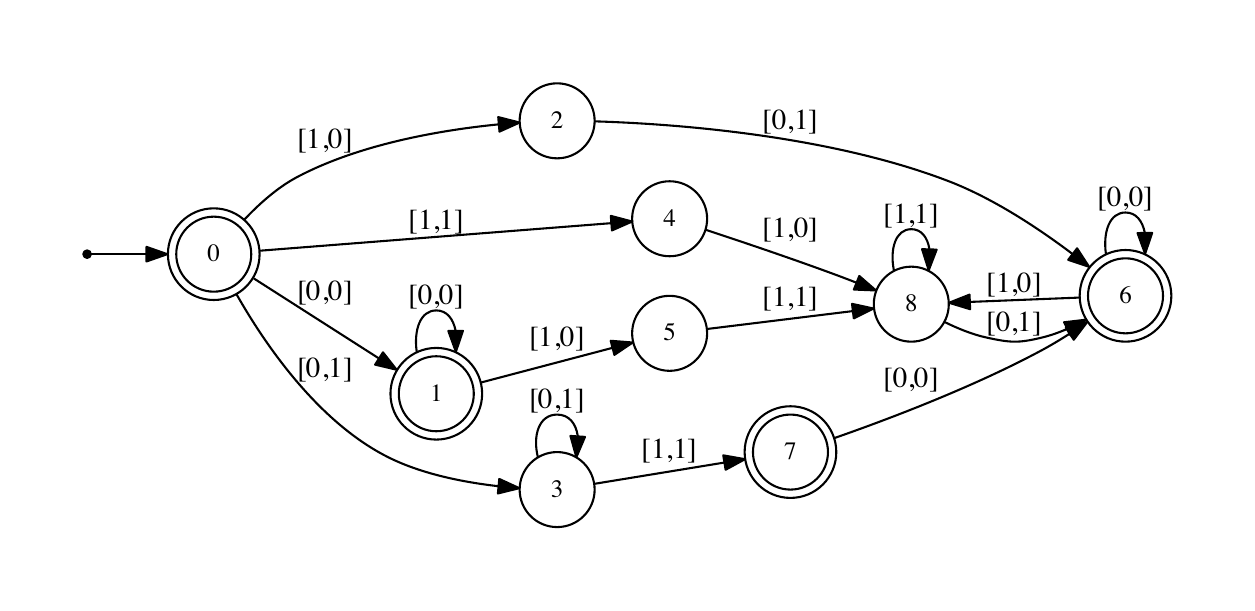}
\caption{Synchronized automaton {\tt ep\_reg} for ending positions
of runs of the regular paperfolding sequence.}
        \label{fig8}
    \end{minipage}
\end{figure}

Once we have these automata, we can easily recover many of the
results of \cite{Bunder&Bates&Arnold:2024}, such as
their Theorem 3.2.   For example they proved that
if $n \equiv 1$ (mod $4$), then $h(n) = 2n$.  We can prove this
as follows with {\tt Walnut}:
\begin{verbatim}
eval test32a "?lsd_2 An (n=4*(n/4)+1) => $ep_reg(n,2*n)":
\end{verbatim}
The reader may enjoy constructing {\tt Walnut} expressions to check
the other results of \cite{Bunder&Bates&Arnold:2024}.

Slightly more challenging to prove is the sum property, conjectured by
Hendriks, and given in
\cite[Thm.~4.1]{Bunder&Bates&Arnold:2024}.  We state it as follows:
\begin{theorem}
Arrange the set of positive integers
not in $H := \{ h(n)+1 \, : \, n \geq 0 \}$ in 
increasing order, and let $t(n)$ be the $n$'th such integer, for $n \geq 1$.
Then
\begin{itemize}
\item[(a)] $g(h(i)+1) = 2$ for $i \geq 0$;
\item[(b)] $g(t(2i)) = 3$ for $i \geq 1$;
\item[(c)] $g(t(2i-1)) = 1$ for $i \geq 1$.
\end{itemize}
\end{theorem}

\begin{proof}
The first step is to create an automaton {\tt tt} computing $t(n)$.  Once again,
we guess the automaton from data and then verify its correctness.
It is depicted in Figure~\ref{fig9}.
\begin{figure}[htb]
\begin{center}
\includegraphics[width=5.5in]{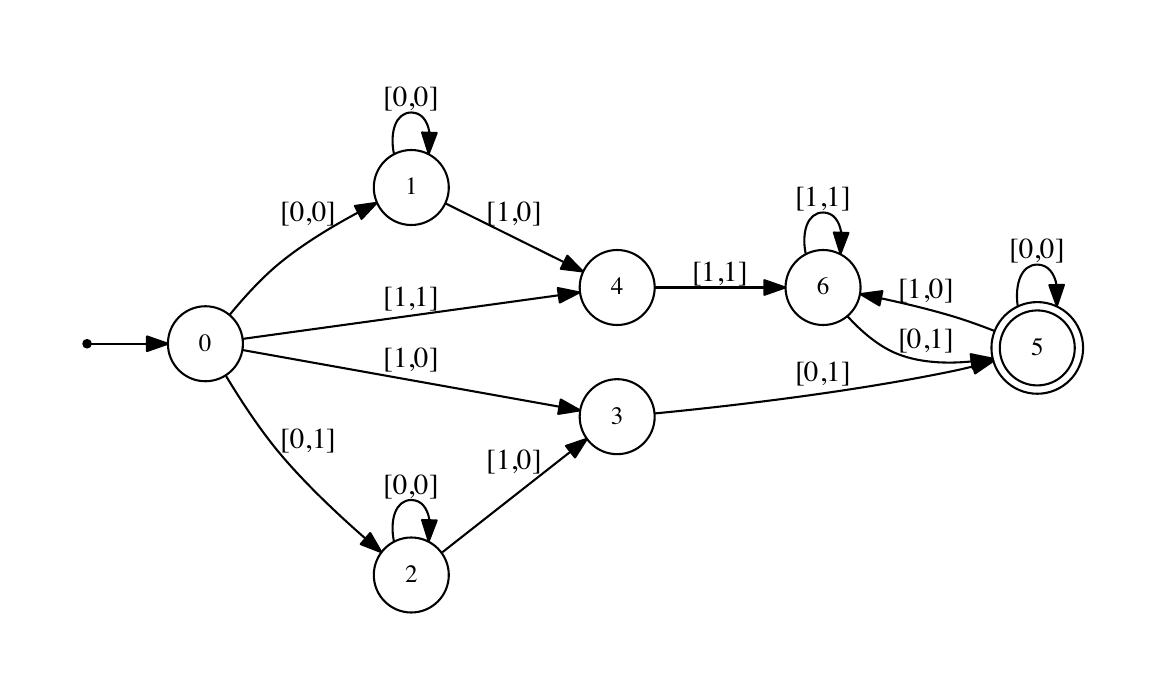}
\end{center}
\caption{The automaton {\tt tt} computing $t(n)$.}
\label{fig9}
\end{figure}

In order to verify its correctness,
we need to verify that {\tt tt} indeed computes a increasing function $t(n)$
and further that the set $\{ t(n) \, : \, n \geq 1 \} =
\{ 1,2, \ldots, \} \setminus H$. We can do this as follows:
\begin{verbatim}
eval tt1 "?lsd_2 An (n>=1) => Ex $tt(n,x)":
# takes a value for all n
eval tt2 "?lsd_2 ~En,x,y n>=1 & x!=y & $tt(n,x) & $tt(n,y)":
# does not take two different values for the same n
eval tt3 "?lsd_2 An,y,z (n>=1 & $tt(n,y) & $tt(n+1,z)) => y<z":
# is an increasing function
eval tt4 "?lsd_2 Ax (x>=1) => 
   ((En n>=1 & $tt(n,x)) <=> (~Em,y $ep_reg(m,y) & x=y+1))":
# takes all values not in H
\end{verbatim}

Now we can verify parts (a)-(c) as follows:
\begin{verbatim}
eval parta "?lsd_2 Ai,x (i>=1 & $ep_reg(i,x)) => RLR[x+1]=@2":
eval partb "?lsd_2 Ai,x (i>=1 & $tt(2*i,x)) => RLR[x]=@3":
eval partc "?lsd_2 Ai,x (i>=1 & $tt(2*i-1,x)) => RLR[x]=@1":
\end{verbatim}
And {\tt Walnut} returns {\tt TRUE} for all of these.  This completes the
proof.
\end{proof}

\section{Connection with continued fractions}

Dimitri Hendriks observed, and Bunder et al.~\cite{Bunder&Bates&Arnold:2024}
proved, a relationship between the sequence of runs for the regular
paperfolding sequence, and the
continued fraction for the real number $\sum_{i \geq 0} 2^{-2^i}$.

As it turns out, however, a {\it much\/} more general result holds; it links
the continued fraction for uncountably many irrational numbers to
runs in the paperfolding sequences.  

\begin{theorem}
Let $n\geq 2$ and $\epsilon_i \in \{ -1, 1\}$ for $2 \leq i \leq n$.
Define
$$\alpha(\epsilon_2, \epsilon_3, \ldots, \epsilon_n) :=
	{1\over 2} + {1 \over 4} + \sum_{2\leq i\leq n} {\epsilon_i}2^{-2^i} .$$
Then the continued fraction for
$\alpha(\epsilon_2, \epsilon_3, \ldots, \epsilon_n)$
is given by $[0, 1, (2R_{1, \epsilon_2, \epsilon_3, \ldots, \epsilon_n})']$,
where the prime indicates that the last term is increased by $1$.

As a consequence, we get that the numbers
$\alpha(\epsilon_2, \epsilon_3,\ldots)$
have continued fraction given by
$[0, 1, 2R_{1, \epsilon_2, \epsilon_3, \ldots}]$.
\label{bba}
\end{theorem}

\begin{remark}
The numbers $\alpha(\epsilon_2, \epsilon_3,\ldots)$ were proved transcendental by Kempner \cite{Kempner:1916}.
They are sometimes erroneously called Fredholm numbers, even though
Fredholm never studied them.
\end{remark}

As an example, suppose $(\epsilon_2,\epsilon_3,\epsilon_4,\epsilon_5) = 
(1,-1,-1,1)$.   Then 
$$x(1,-1,-1,1) = 3472818177/2^{32} = [0, 1, 4, 4, 2, 6, 4, 2, 4, 4, 6, 4, 2, 4, 6, 2, 4, 5],$$ while
$R_{1,1,-1,-1, 1} = 2213212232123122$.

To prove Theorem~\ref{bba}, we need the ``folding lemma'':
\begin{lemma}
Suppose $p/q = [0, a_1, a_2,\ldots, a_t]$, $t$ is odd, and
$\epsilon \in \{-1, 1\}$.  Then
$$p/q + \epsilon/q^2 = [0, a_1, a_2, \ldots, a_{t-1}, a_t - \epsilon,
a_t + \epsilon, a_{t-1}, \ldots, a_2, a_1].$$
\label{lem1}
\end{lemma}

\begin{proof}
See \cite[p.~177]{Dekking&MendesFrance&vanderPoorten:1982}, although
the general ideas can also be found in
\cite{Shallit:1979,Shallit:1982b}.
\end{proof}

We can now prove Theorem~\ref{bba} by induction.
\begin{proof}
From Lemma~\ref{lem1} we see that
if $\alpha(\epsilon_2, \epsilon_3, \ldots, \epsilon_n) 
= [0, 1, a_2, \ldots, a_t]$
then
$$\alpha(\epsilon_2, \epsilon_3, \ldots, \epsilon_n, \epsilon_{n+1})
= [0, 1, a_2, \ldots, a_{t-1}, a_t - \epsilon_{n+1},
a_t + \epsilon_{n+1}, a_{t-1}, a_{t-2}, \ldots, a_3, a_2+1] .$$

Now
$F_{1, \epsilon_2, \epsilon_3, \ldots, \epsilon_n}$
always ends in $-1$.
Write
$R_{1, \epsilon_2, \epsilon_3, \ldots, \epsilon_n} 
= b_1 b_2 \cdots b_t$.  Then
$$R_{1, \epsilon_2, \ldots, \epsilon_n, \epsilon_{n+1}}
= b_1 \cdots b_{t-1}, b_t +1, b_{t-1}, \ldots, b_1$$
if $\epsilon_{n+1} = -1$ (because we extend the
last run with one more $-1$) and
$$R_{1, \epsilon_2, \ldots, \epsilon_n, \epsilon_{n+1}}
= b_1 \cdots b_{t-1}, b_t, b_t+1, b_{t-1}, \ldots, b_1$$
if $\epsilon_{n+1} =1$.

Suppose
\begin{align*}
\alpha(\epsilon_2, \epsilon_3, \ldots, \epsilon_n) &= 
[0, 1, (2R_{1, \epsilon_2, \epsilon_3, \ldots, \epsilon_n})'] \\
&= [0, 1, a_2, \ldots, a_t ],
\end{align*}
and let
$R_{1, \epsilon_2, \ldots, \epsilon_n} = b_1 b_2 \cdots b_{t-1}$.
Then
\begin{align*}
\alpha(\epsilon_2, \epsilon_3, \ldots, \epsilon_n, \epsilon_{n+1} )
&= [0, 1, a_2, \ldots, a_{t-1}, a_t - \epsilon_{n+1}, a_t + \epsilon_{n+1},
	a_{t-1}, \ldots, a_3, a_2+1]  \\
&= [0, 1, 2b_1, \ldots, 2b_{t-2}, 2b_{t-1} + 1 - \epsilon_{n+1},
2b_{t-1} + 1 + \epsilon_{n+1}, 2b_{t-2}, \ldots, 2b_2, 2b_1,1] \\
&= [0, 1, 2b_1, \ldots, 2b_{t-2}, 2b_{t-1} + 1 - \epsilon_{n+1}, 
2b_{t-1} + 1 + \epsilon_{n+1}, 2b_{t-2}, \ldots, 2b_2, 2b_1 + 1]  \\
&= [0, 1, 2R_{1, \epsilon_2, \ldots, \epsilon_{n+1}}'] ,
\end{align*}
as desired.
\end{proof}

\section{Acknowledgments}

I thank Jean-Paul Allouche and Narad Rampersad for their helpful comments.

\end{document}